\documentclass [12pt,draft]{article}
\usepackage[T2A]{fontenc}
\usepackage{amsbsy}
\usepackage{setspace}
\usepackage{amsmath, amsfonts, amssymb, amsthm, euscript, amscd, latexsym, mathrsfs, bm, graphicx}
\newcommand{\R}{{\rm I}\kern-0.18em{\rm R}}

\input{amssym.def}

\def\RR{\mathbb{R}}

\renewcommand{\C}{{\rm C}}

\newcommand{\e}{{\sss E}}

\newtheorem{Thm}{Theorem}[section]
\newtheorem{Prop}[Thm]{Proposition}
\newtheorem{Lem}[Thm]{Lemma}
\newtheorem{lem}[Thm]{Lemma}

\def\aa#1{ \begin{align*} #1 \end{align*} }
\def\aaa#1{ \begin{align} #1 \end{align} }
\def\mm#1{ \begin{multline*} #1 \end{multline*} }
\def\mmm#1{ \begin{multline} #1 \end{multline} }

 \newcommand{\<}{\langle}
\renewcommand{\>}{\rangle}

\newcommand{\sss}{\scriptscriptstyle}

\newcommand{\eps}{\varepsilon}
\newcommand{\pl}{\partial}
\renewcommand{\i}{{(i)}}
\renewcommand{\j}{{(j)}}
\newcommand{\gt}{\geqslant}
\newcommand{\lt}{\leqslant}
\newcommand{\te}{\theta}
\newcommand{\rf}{\eqref}

\newcommand{\dl}{\delta}
\newcommand{\al}{\alpha}

 \newcommand{\Dl}{\Delta}
 \newcommand{\la}{\lambda}
 
 \newcommand{\sg}{\sigma}

\newcommand{\om}{\omega}
\newcommand{\mc}{\mathcal}
\newcommand{\Om}{\Omega}

\newcommand{\x}{\times}
\newcommand{\mto}{\mapsto}

\newcommand{\nab}{\nabla}
\newcommand{\fdot}{\,\cdot\,}
\newcommand{\td}{\tilde}
\newcommand{\lb}{\label}
\def\Rnu{{\mathbb R}}

\renewcommand{\leq}{\lt}

\def\ffi{\varphi}

\def\til{\tilde}
\numberwithin{equation}{section}

\begin{document}

\noindent{\bf\sl Global and Stochastic Analysis\\
Vol. 2, No. 1, June 2012\\
ISSN 2248-9444\\
Copyright\copyright   Mind Reader Publications}

\vspace{0.5cm}

\addcontentsline{toc}{section}{R.S. Pereira and E. Shamarova.
Forward-Backward SDEs driven by L\'evy Processes and Application
 to n}

\begin{center}
{\LARGE \bf Forward-Backward SDEs \protect\\\vspace{2.5mm}
driven by L\'evy Processes and\protect\\
\vspace{5mm} Application to
Option Pricing}\\

\vspace{5mm} {\bf R.S. Pereira$^a$ and E. Shamarova$^b$}\\

 \vspace{3mm}

 Centro de Matem\'atica da Universidade do Porto\\
 Rua do Campo Alegre 687, 4169-007 Porto, Portugal\\
 $^a$manuelsapereira@gmail.com \\
 $^b$evelinas@fc.up.pt

\vspace{2mm}

{\small Received by the Editorial Board on May 2, 2012}

\vspace{1mm}

\end{center}

\begin{abstract}
Recent developments on financial markets have revealed the limits of
Brownian motion pricing models when they are applied to
actual markets. L\'evy processes, that admit jumps over time, have
been found more useful for applications. Thus, we suggest
a L\'evy model based on Forward-Backward Stochastic Differential
Equations (FBSDEs) for option pricing in a L\'evy-type market. We
show the existence and uniqueness of a solution to FBSDEs driven by
a L\'evy process. This result is important from the mathematical
point of view, and also, provides a much more realistic approach to
option pricing.

\vspace{2mm}


\vspace{2mm}


\end{abstract}

\section*{1 Introduction}

\setcounter{section}{1}

\setcounter{equation}{0}

\setcounter{Thm}{0}

Since the seminal contribution made by F. Black and M. Scholes
\cite{bs}, several methodologies to value contingent assets have
been developed. From Plain Vanilla options to complex instruments
such as Collateral Debt Obligations or Baskets of Credit Default
Swaps, there are models to price virtually any type of contingent
asset. There are, indeed, successful attempts on providing general
models which in theory could price any kind of contingent claim
(see, for example \cite{cvitanic}), given a payoff function. The
idea behind these models is quite standard: A portfolio replicating
the payoff function of the asset is devised and, under non-arbitrage
conditions, the price of the asset at a certain instant of time is
the price of this portfolio at that time.
 
However, in spite of all this diversity and sophistication,
there is an assumption that is, up to recent times, rarely
questioned. Specifically, we refer to the assumption that stock
prices are continuous diffusion processes,
presupposing thereby that the returns have normal distributions at any time. However, it is well known today
that empirical distributions of stock prices returns tend to deviate
from normal distributions, either due to skewness, kurtosis or even
the existence of discontinuities (Eberlein et al. give evidence of
this phenomenae in \cite{eberlein}).
The recent developments have shown that the reliance on normal
distribution can bring costly surprises, especially when extreme and
disruptive events occur with a much higher frequency than the one
estimated by models.

As such, we believe that no matter which historical status the normal distribution
has acquired throughout the years, strong efforts should be
undertaken in order to develop alternative models that incorporate
assumptions adequated to the observed evidence on financial markets,
such as asymmetry or skewness. We do not pretend that some
definitive model can actually be developed, especially when market
participants' main activities are currently shifting due to the
conditions imposed on financial markets. Indeed,
the recently introduced regulations on financial markets severely
restraining the use of own's capital for trading purposes will force
the financial players to find new ways of driving a profit. This
adds another layer of uncertainty about the assumptions imposed on a
model. We believe, however, that in spite of the inherent inability
to prove that any present model can account for future market
conditions, it is worth to attempt to correctly price financial
claims in the present and near future market conditions, which, as
it is clear now, is fundamental to the stability of markets.

Taking the above considerations into account,  we propose to
replicate contingent claims in L\'evy-type markets, i.e. in markets
with the stock-price dynamics described as $S_t=S_0\, e^{X_t},$
where $X_t$ is a L\'evy-type stochastic integral defined in
\cite{applebaum}. This  allows the relaxation of conditions posed on
the pricing process such as symmetry, non-skewness or continuity,
imposed by the Brownian framework. The self-similiarity of the
pricing process, appearing due to a Brownian motion, is also ruled
out from the assumptions. We base our model on the study of
Forward-Backward Stochastic Differential Equations (FBSDEs) driven
by a L\'evy process. FBSDEs combine equations with the initial and
final conditions which allows one to search for a replicating
portfolio.
 Specifically, we are concerned with the following fully coupled FBSDEs:
\begin{equation}
\label{fbsde}
\begin{cases}
X_{t} = x + \int_{0}^{t} f(s, X_{s}, Y_{s},Z_{s})\: ds +
\sum_{i=1}^\infty\int_{0}^{t} \sigma_i(s,X_{s-}, Y_{s-})\:
dH^\i_{s},
\\
Y_{t} = h(X_{T}) + \int_{t}^{T} g(s, X_{s}, Y_{s},Z_{s})\: ds -
\sum_{i=1}^\infty\int_{t}^{T} Z^i_{s} \: dH^\i_{s},
\end{cases}
\end{equation}
where the stochastic integrals are written with respect to the
orthogonalized Teugels martingales $\{H^\i_t\}_{i=1}^\infty$
associated with a L\'evy process $L_t$ \cite{WSchoutens}. We are
searching for an $\Rnu^P\x \Rnu^Q \x (\Rnu^Q\x \ell_2)$-valued
solution $(X_t,Y_t,Z_t)$ on an arbitrary time interval $[0,T]$,
which is square-integrable and adapted with respect to the
filtration $\mc F_t$ generated by $L_t$. To the authors' knowledge,
fully coupled FBSDEs of this type have not been studied before.
Fully decoupled FBSDEs involving L\'evy processes as drivers were
studied by Otmani \cite{Otmani}. Backward SDEs driven by Teugels
martingales were studied by Nualart and Schoutens \cite{nualart}.
Our method of solution to the FBSDEs could be compared to the Four
Step Scheme  \cite{Ma-Protter}. The original four step scheme deals
with FBSDEs driven by a Brownian motion, and the solution is
obtained via the solution to a quasilinear PDE. Replacing the
stochastic integral with respect to a Brownian motion with a
stochastic integral with respect to the orthogonalized Teugels
martingales leads to a partial integro-differential equation (PIDE).
The solution to the PIDE is then used to obtain the solution to the
FBSDEs.

The organization of the paper is as follows. In Section 2, we give
some preliminaires on the martingales $\{H^\i_t\}$. In Section 3,
under certain assumptions, we obtain the existence and uniqueness
result for the associated PIDE. Our main result is Theorem
\ref{main-result}, where we obtain a solution to FBSDEs \rf{fbsde}
via the solution to the PIDE and prove its uniqueness. In section 4,
we apply the results of Section 3 to model hedging options for a
large investor in a L\'evy-type market. Previously, this problem was
studied by Cvitanic and Ma \cite{cvitanic} for a Brownian market
model. Finally, we study conditions for the existence of replicating
portfolios.

\section*{2 Preliminaires}

\setcounter{section}{2}

\setcounter{equation}{0}

\setcounter{Thm}{0}

 Let $(\Omega, \mc F, \mc F_t,P)$ be a filtered probability space,
where $\{ {\mc F}_t \}$, $t \in [0,T]$, is the filtration generated
by a real-valued L\'evy process $L_t$. Note that the L\'evy measure
$\nu$ of $L_t$ always satisfies the condition \aa{ \int_\Rnu
(1\wedge x^2) \, \nu(dx) < \infty. } We make the filtration  $\mc
F_t$ $P$-augmented, i.e. we add all $P$-null sets of $\mc F$ to each
$\mc F_t$. Following Nualart and Schoutens~\cite{nualart1} we
introduce the orthogonalized Teugels martingales
$\{H_t^{(i)}\}_{i=1}^{\infty}$  associated with $L_t$. For this we
assume that for every $\eps>0$, there exists a $\la>0$ so that \aa{
\int_{(-\eps,\eps)^c} \exp(\la|x|)\, \nu(dx) < \infty. } The latter
assumption guaranties that \aa{ \int_\Rnu |x|^i \, \nu(dx) < \infty
\quad \text{for} \; i=2,3,\ldots. } It was shown in
\cite{WSchoutens} that under the above assumptions one can introduce
the power jump processes and the related Teugels maringales.
Futhermore, it was shown that the strong orthogonalization procedure
can be applied to the Teugels martinagles and that the
orthonormalization of the Teugels martingales corresponds to the
orthonormalization of the polynomials $1,x, x^2, \ldots$ with
respect to the measure $x^2\nu(dx) + a^2\dl_0(dx)$, where the
parameter $a\in\Rnu$ is defined in Lemma \ref{hi-rep} below.
As in \cite{nualart}, by  $\{q_{i}(x)\}$ we denote the system of
orthonormalized polynomials such that $q_{i-1}(x)$ corresponds to
$H^\i_t$. Also, we define the polynomial $p_i(x) = x q_{i-1}(x)$. We
refer to \cite{WSchoutens} for details on the Teugels martingales
and their orthogonalization procedure. In the following, Lemma
\ref{hi-rep} below will be usefull.
\begin{Lem}
\lb{hi-rep} The process $H^\i_t$ can be represented as follows: \aa{
H_t^\i= q_{i-1}(0) B^\lambda(t) + \int_\RR p_i(x)\tilde{N}(t,dx), }
 where $B^\lambda(t) = \sum_{i=1}^N \lambda_i B_i (t)$ with $\la^T\la=a$, $\la_i\in\Rnu$,
 $\{B_i(t)\}_{i=1}^N$ are independent real-valued Brownian motions, and $\td N(t,A)$ is the compensated
 Poisson random measure that corresponds to the Poisson point process $\Dl L_t$.
\end{Lem}
\begin{proof}
We will use the representation below for $H^\i_t$ obtained in
\cite{nualart}: \aa{ H_t^\i = q_{i-1}(0)L_t + \sum_{0 < s \leq t}
\tilde{p}_i(\Delta L_s) - t E \bigg[ \sum_{0 < s \leq 1}
\tilde{p}(\Delta L_s)\bigg] - tq_{i-1}(0)E[L_1], } where $\td p_i(x)
= p_i(x) - xq_{i-1}(0)$, and $E$ is the expectation with respect to
$P$.
 Taking into account  that $L_t= L_t^c + \sum_{0 \lt s \lt t} \Delta L_s$, where $L^c_t$
 is the continuous part of $L_t$, we obtain:
 \aa{
 H_t^\i&= q_{i-1}(0)L_t^c + \sum_{0 < s \lt t}p_i(\Delta L_s)
 - t E\bigg[\sum_{0 < s \lt 1} \tilde{p}_i (\Delta L_s)\bigg] -tq_{i-1}(0)E[L_1]\\
 &= q_{i-1}(0)\Big[L_t^c - E[L_t^c]\Big] + \sum_{0 < s \lt t}p_i(\Delta L_s) - E\bigg[ \sum_{0 < s \lt t}p_i(\Delta L_s)\bigg]\\
 &= q_{i-1}(0)B^\lambda (t) +  \int_\RR p_i(x) \tilde{N} (t, dx).
 }
\end{proof}
In the sequence, the following lemma will be frequently applied:
\begin{lem}
It holds that \label{pi-pro} \aa{ \int_\Rnu p_i(x) p_j(x) \nu(dx) =
\dl_{ij} - a^2 q_{i-1}(0)q_{j-1}(0). }
\end{lem}
\begin{proof}
The proof is a straightforward corollary of the orthonormality of
$q_{i-1}(x)$ with respect to the measure $x^2\nu(dx) +
a^2\dl_0(dx)$.
\end{proof}
We will need an analog of Lemma 5 from \cite{nualart} which was proved in the latter article 
for a pure-jump $L_t$.  We obtain this result for the case when
$L_t$ has both the continuous and the pure-jump parts.
\begin{Lem}
\lb{lemma5}
Let $h: \Omega \times [0,T] \x \Rnu \rightarrow \Rnu^n$ be a random
function
satisfying \aaa{ E \int_0^T |h(s,y)|^2 \nu(dy)< \infty.
\label{condition5} }
 Then, for each $t \in [0,T]$,
\aa { \sum_{t < s \leq T} h(s,& \Delta L_s) = \sum_{i=1}^{\infty}
\int_t^T \int_\Rnu \nu(dy) h(s,y)p_i(x)\, dH_s^\i+\int_t^T \int_\RR
h(s,y) \nu(dy) ds. }
\end{Lem}
\begin{proof}
Note that \aaa{ \lb{Mt} M_t= \sum_{0 \leq s \leq t} h(s,\Delta L_s)
- \int_0^t \int_\RR h(s,y) \nu(dy) ds = \int_0^t\int_\Rnu h(s,x)\,
\td N(ds,dx) } is a square integrable martingal, i.e. $\sup_{t\in
[0,T]} E|M_t|^2 < \infty$, by \rf{condition5}. By the predictable
representation theorem \cite{nualart1}, there exist predictable
processes $\ffi_i$ with $E\Big[\int_0^T \sum_{i=1}^\infty
|\ffi_i|^2\Big]<\infty$ and such that $M_t = \sum_{i=1}^{\infty}
\int_0^t \ffi_i(s) dH^\i_s$. Since $\left\langle H^\i,H^\j
\right\rangle_t = t\,\delta_{ij}$ \cite{WSchoutens}, then \aa{
\left\langle M,H^\i \right\rangle_t= \int_0^t \ffi_i(s) ds. } On the
other hand, by \rf{Mt} and Lemma \ref{hi-rep}, \aa{
&\< M,H^\i \>_t\\
&=\left<\ \int_0^t \int_{\RR} h(s,x) \td N(ds,dx),
q_{i-1}(0)B^\la_t + \int_0^t \int_{\RR} p_i(x)\tilde{N}(dt,dx) \right\>_t \\
&= \int_0^t\int_\RR h(s,x) p_i(x)\nu(dx) ds. }
 This implies that $\ffi_i(s)= \int_\RR h(s,y)p_i(y) \nu(dy)$,
 and therefore,
 \aa{
 \sum_{0 \leq s \leq t} h(s,\Delta L_s) - \int_0^t \int_\RR h(s,y) \nu(dy) ds =
 \sum_{i=1}^{\infty}\int_0^t\int_\RR &h(s,y)p_i(y) \nu(dy)dH_s^\i.
 }
 \end{proof}


%
\section*{3 FBSDEs and the associated PIDE}

\setcounter{section}{3}

\setcounter{equation}{0}

\setcounter{Thm}{0}

\subsection*{3.1 Problem Formulation and Assumptions}

Consider the FBSDEs: \aaa{ \lb{fbsde1}
\begin{cases}
X_{t} = x + \int_{0}^{t} f(s, X_{s}, Y_{s},Z_{s})\: ds +
\int_{0}^{t} \sigma(s,X_{s-}, Y_{s-})\: dH_{s},
\\
Y_{t} = h(X_{T}) + \int_{t}^{T} g(s, X_{s}, Y_{s},Z_{s})\: ds - \int_{t}^{T} Z_{s} \: dH_{s}, \\
t\in [0,T],
\end{cases}
} where \aa{
&f: [0,T] \x \Rnu^P \x \Rnu^Q \x (\Rnu^Q\x \ell_2) \to \Rnu^P,\\
&\sg:[0,T] \x \Rnu^P \x \Rnu^Q \to \Rnu^P \x \ell_2, \\
&g: [0,T] \x \Rnu^P \x \Rnu^Q \x (\Rnu^Q\x \ell_2) \to \Rnu^Q,\\
&h: \Rnu^P \to \Rnu^Q } are Borel-measurable functions. Here, the
stochastic integrals \aa{ \int_{0}^{t} \sigma(s,X_{s-}, Y_{s-})\,
dH_{s} \quad \text{and} \quad \int_{t}^{T} Z_{s} \,dH_{s} } are
shorthand notation for \aa{ \sum_{i=1}^\infty \int_{0}^{t}
\sg_i(s,X_{s-}, Y_{s-})\, dH^\i_{s}\quad \text{and} \quad
\sum_{i=1}^\infty\int_{t}^{T} Z^i_{s} \, dH^\i_{s} } respectively,
where $Z_s = \{Z^i_s\}_{i=1}^\infty$, $\sg =
\{\sg_i\}_{i=1}^\infty$, $\sg_i:[0,T] \x \Rnu^P \x \Rnu^Q \to
\Rnu^P$. The solution to FBSDEs \rf{fbsde1}, when exists, will be an
$\Rnu^P \x \Rnu^Q \x (\Rnu^Q \x \ell_2)$-valued $\mc F_t$-adapted
triple $(X_t,Y_t,Z_t)$ satisfying \aa{ E\int_0^T \Big(|X_t|^2 +
|Y_t|^2 + \sum_{i=1}^\infty |Z^i_t|^2\Big) \,dt <\infty, } and
verifying \rf{fbsde1} $P$-a.s.. The latter includes the existence of
the stochastic integrals in \rf{fbsde1}. Implicitly, we are assuming
that $X_t$ and $Y_t$ have left limits, and that $Z_t$ is $\mc
F_t$-predictable. So in fact, we are searching for c\`adl\`ag
$(X_t,Y_t)$, which will guarantee the existence of $X_{t-}$ and
$Y_{t-}$, and predictable $Z_t$.

We associate to \rf{fbsde1} the following final value problem for a
PIDE:
 \begin{gather}
\label{PIDE}
\left\{
\begin{array}{l}
\partial_t \theta(t,x) +  f^k(t,x,\theta(t,x),\te^{(1)}(t,x))\, \partial_{k} \theta(t,x)
+ \beta^{kl}(t,x,\theta(t,x))\, \partial^2_{kl}\, \theta(t,x)\\
- \int_\Rnu \big[\te\big(t,x + \delta(t,x,\theta(t,x),y)\big) - \te(t,x)- \pl_k \te(t,x)\dl^k(t,x,\theta(t,x),y)\big]\nu(dy)\\
 +g(t, x,\theta(t,x),\te^{(1)}(t,x)) = 0,\\
\theta(T,x) = h(x)
\end{array}\right.
\end{gather}
 with $\te^{(1)}: [0,T] \x \Rnu^P \to \Rnu^Q \x \ell_2$, \mmm{
\lb{te1}
 \te^{(1)}_i(t,x)=
 \int_\Rnu [\theta\big(t,x + \dl(t,x,\theta(t,x),y)\big) - \te(t,x)] p_i(y)\, \nu(dy)\\
 +  c^k_i(t,x,\te(t,x))\,\pl_{k}\te(t,x).
} The connection between $\beta^{kl}$, $\dl$, $c^k_i$ and the
coefficients of FBSDEs \rf{fbsde1} is the following: \aaa{
&\dl(t,x,y,y') = \sum_{i=1}^\infty \sg_i(t,x,y)p_i(y'),\lb{dl}\\
& \beta^{kl}(t,x,y) = \frac{a^2}2\,\Big(\sum_{i=1}^\infty
\sg^k_i(t,x,y) q_{i-1}(0)\Big)
\Big( \sum_{j=1}^\infty \sg^l_j(t,x,y) q_{j-1}(0)\Big),\lb{akl}\\
& c^k_i(t,x,y) = \sg^k_i(t,x,y) - \int_\Rnu \dl^k(t,x,y,y')
p_i(y')\,\nu(dy').\lb{cik} } To guarantee the existence of the above
functions we will make the assumption:
\begin{itemize}
\item[\bf A0] \quad
$\sum_{i=1}^\infty q_{i-1}(0)^2 < \infty$.
\end{itemize}
Since $\sg^k = \{\sg^k_i\}_{i=1}^\infty$ takes values in $\ell_2$,
A0 immediately guarantees the convergence of the both multipliers in
\rf{akl}. The convergence of the series in \rf{dl} is understood  in
$L_2(\nu(dy'))$ for each fixed $(t,x,y)$. Moreover, it holds that
\mmm{ \lb{dl-claim} \int_\Rnu\Big|\sum_{i=1}^\infty
\sg_i(t,x,y)p_i(y')\Big|^2\nu(dy') \\= \|\sg(t,x,y)\|^2_{\Rnu^P\x
\ell_2} - a^2 \Big|\sum_{i=1}^\infty \sg_i(t,x,y)\,
q_{i-1}(0)\Big|^2. } Indeed, applying Lemma \ref{pi-pro} for each
fixed $N$, we obtain: \mm{ \int_\Rnu \Big|\sum_{i=1}^N
\sg_i(t,x,y)p_i(y')\Big|^2 \nu(dy')
= \sum_{i,j=1}^N (\sg_i,\sg_j) \int_\Rnu p_i(y')p_j(y')\nu(dy')\\
= \sum_{i=1}^N |\sg_i|^2 - a^2 \Big|\sum_{i=1}^N \sg_i
q_{i-1}(0)\Big|^2. } Now letting $N$ tend to infinity, we obtain
\rf{dl-claim}.
\begin{lem}
\lb{assers} The following assertions hold:
\begin{enumerate}
\item \lb {cik1} $c^k_i(t,x,y) = a^2\, q_{i-1}(0) \sum_{j=1}^\infty \sg^k_j(s,x,y)q_{j-1}(0)$.
\item \lb{asr2} For each $k$, $c^k = \{c^k_i\}_{i=1}^\infty$ takes values in $\ell_2$.
\item \lb {asr3} For each $(s,x,y)$, $\Big\{\int_\Rnu \dl(s,x,y,y')p_i(y')\, \nu(dy')\Big\}_{i=1}^\infty \in \ell_2$.
\end{enumerate}
\end{lem}
\begin{proof}
Define $\dl_N(s,x,y,y')= \sum_{j=1}^N \sg_j(s,x,y)p_j(y')$. By what
was proved, for each $(s,x,y)$, $\dl_N(s,x,y,\fdot) \to
\dl(s,x,y,\fdot)$ in $L_2(\nu(dy'))$, and therefore, for each $i$,
and for each $(s,x,y)$, \aa{ \int_\Rnu
\dl_N(s,x,y,y')p_i(y')\nu(dy') \to \int_\Rnu
\dl(s,x,y,y')p_i(y')\nu(dy') } as $N\to\infty$. On the other hand,
by Lemma \ref{pi-pro}, \aa{ \int_\Rnu \dl_N(s,x,y,y') p_i(y')
\nu(dy') = \sg_i(s,x,y) - a^2\, q_{i-1}(0) \sum_{j=1}^N
\sg_j(s,x,y)q_{j-1}(0). } Comparing the last two relations, we
obtain that \aaa{ \lb{dlpi} \int_\Rnu \dl(s,x,y,y')p_i(y')\nu(dy') =
\sg_i(s,x,y) - a^2\, q_{i-1}(0) \sum_{j=1}^\infty
\sg_j(s,x,y)q_{j-1}(0) } which proves Assertion \ref{cik1}.
Assertion \ref{asr2} is implied by Assumption A0 and Assertion
\ref{cik1}. Finally, \rf{cik} implies Assertion \ref{asr3}.
\end{proof}
The heuristic argument behind PIDE \eqref{PIDE} assumes the
connection $Y_t = \theta(t,X_t)$ between the solution processes
$X_t$ and $Y_t$ to \eqref{fbsde1} via a $\C^{1,2}$-function
$\theta$. It\^o's formula applied to $\theta(t,X_t)$ at points $t$
and $T$ leads to another BSDE which has to be the same as the given
BSDE in \eqref{fbsde1}. Thus we ``guess'' PIDE \rf{PIDE} by equating
the drift and stochastic terms of these two BSDEs.

\subsection*{3.2 Solvability of the PIDE}

We solve Problem \rf{PIDE} for a particular case when
the functions $f(t,x,y,z)$ and $g(t,x,y,z)$ do not depend on $z$,
and for a short time duration $T$.
Thus, we are dealing with the following final value problem for a
PIDE:
\aaa{ \label{evolution}
\begin{cases}
&\pl_t \te(t, x) = -[A(t,\te(t,\cdot))\te](x) + g(t,x,\te(t,x)), \\ 
&\te(T,x) = h(x),
\end{cases}
} where $A(t,\rho(t,\cdot))$ is a partial integro-differential
operator given by
\mmm{
\lb{generator}
[A(t,\rho(t,\cdot))\te](x)  
=f^k(t,x,\rho(t,x)) \, \pl_{k}
\te(t,x)
+ \beta^{kl}(t,x,\rho(t,x))\, \pl^2_{kl}\, \te(t,x)\\
+ \int_\Rnu \big[\te\big(t,x + \dl(t,x,\rho(t,x),y)\big)-\te(s,x)
-\pl_k \te(t,x)\dl^k(t,x,\rho(t,x),y)\big] \, \nu(dy).
}
with the domain $D(A(t,\rho(t,\cdot)))=\C_b^{2}(\Rnu^P\to\Rnu^Q)$,
the space of bounded continuous functions $\Rnu^P\to\Rnu^Q$ whose
first and second order derivatives are also bounded. 
We assume the following:
\begin{enumerate}
\item[\bf A1] Functions $f$, $g$, $\sg$, and $h$ are bounded and have bounded 
spatial derivatives of the first and the second order.
\end{enumerate}
\begin{lem}
Let A0 and A1 be fulfilled. Then
$A(t,\rho(t,\cdot))$, defined by \rf{generator}, is a generator of a
strongly continuous semigroup on $\C_b(\Rnu^P\to\Rnu^Q).$
\end{lem}
\begin{proof}
Note that by Assertion \ref{cik1} of Lemma \ref{assers} and by
\rf{akl}, functions $c^k$ and $\beta^{kl}$ are bounded and Lipschitz
in the spatial variables. This implies that the SDE \aaa{ \lb{SDE4g}
dX^k_s = f^k(t,X_s,\rho(t,X_s)) +
\sum_{i=1}^\infty \sg^k_i(t,X_{s-},\rho(t,X_{s-}))dH^\i_s, } with
$X_t=x$, has a pathwise unique c\`adl\`ag adapted solution on $[t,T]$. The
existence and uniqueness of a solution to an SDE of type \rf{SDE4g}
will be proved in Paragraph 3.3. Now application of It\^o's
formula to $\ffi(X_s)$, where $\ffi$ is twice continuously
differetiable, shows that the operator \rf{generator} 
is the generator of the solution to SDE \rf{SDE4g}, and therefore,
it generates a strongly continuous semigroup on $\C_b(\Rnu^P\to
\Rnu^Q)$. 
\end{proof}
The common method to deal with problems of type \eqref{evolution} is
to fix a $\C_b^{1,2}$-function $\rho(t,x)$, and consider the
following non-autonomous evolution equation: \aaa{
\label{evolution1}
\begin{cases}
&\pl_t \te(t, x) = -[A(t,\rho(t,\cdot))\te](x) - g(t,x,\rho(t,x)),\\
&\te(T,x) = h(x).
\end{cases}
} 
By Assumption A1 and the results of \cite{Tanabe} and \cite{Guli}, there
exists a backward propagator $U(s,t,\rho)$, $0\lt s \lt
t \lt T$, so that \aa{ \te(t,x) = [U(t,T,\rho)h](x) + \int_t^T
[U(t,s,\rho)g(s,\fdot,\rho(s,\fdot))](x)\, ds. } We organize the
map \aaa{ \lb{map-phi} \Phi: \C_b([0,T]\x \Rnu^P\to \Rnu^Q)\to
\C_b([0,T]\x\Rnu^P\to \Rnu^Q),\; \rho \mto \theta, } and prove the
existence of a fixed point. 

Define $E= \C_b(\Rnu^P\to\Rnu^Q)$ and $D=\C_b^{2}(\Rnu^P\to\Rnu^Q)$.
\begin{lem}
\lb{fplem}
Let Assumptions A0 and A1 hold. Then,
there exists a constant $K>0$ that does not
depend on $s$, $t$, $\rho$, and $\rho'$, so that for any function $\ffi\in D$, 
\aa{
\sup_{s\in [t,T]} \|U(t,s,\rho)\ffi - U(t,s,\rho')\ffi\|_E
\lt K \, T \hspace{-1mm} \sup_{s\in [t,T]}\|\rho(s,x) - \rho'(s,x)\|_E\, \|\ffi\|_{D}.
}
\end{lem}
\begin{proof}
We have:
\mm{
\big(U(t,s,\rho') - U(t,s,\rho)\big)\ffi=
\left.U(t,r,\rho')U(r,s,\rho)\ffi\right|_{r=t}^s\\
= \int_t^s dr\, U(t,r,\rho')\big(A(r,\rho'(r,\fdot))-A(r,\rho(r,\fdot))\big)U(r,s,\rho)\ffi.
}
This implies:
\mmm{
\lb{02-08}
\sup_{s\in [t,T]} \|U(t,s,\rho')\ffi - U(t,s,\rho)\ffi\|_E 
\lt T  \sup_{s\in [t,T]}\|U(t,s,\rho')\|_{\mc L(E)}\\ 
\x \sup_{\substack{r,s\in [t,T],\\ r\lt s}} \|U(r,s,\rho)\|_{\mc L(D)} 
\sup_{s\in [t,T]}\|A(s,\rho(s,\fdot))-A(s,\rho'(s,\fdot))\|_{\mc L(D,E)}\,\|\ffi\|_D.
}
Taking into account that \aa{
\|\te(s,\fdot)\|_D = \sup_{x\in\Rnu^P}|\te(s,x)| +
\sup_{x\in\Rnu^P}|\nab \te(s,x)| + \sup_{x\in\Rnu^P}|\nab \nab
\te(s,x)|, } and applying  \rf{akl}, \rf{generator}, and Lemma
\ref{assers},  we obtain that there exists a constant $\bar K>0$ which
does not depend on $s$, $\rho$, and $\rho'$, so that
\aaa{
\lb{gen-diff} &\sup_{\|\te\|_D \lt 1}
\sup_{x\in\Rnu^P}\|A(s,\rho(s, x))\te - A(s,{\rho'}(s, x))\te\|_{\mc L(D,E)}
\notag\\
&\lt \bar K\sup_{x\in\Rnu^P}\Big[ |f(s,x,\rho(s,x))-f(s,x,\rho'(s,x))|
\notag\\
&+ \|\sg(s,x,\rho(s,x))-\sg(s,x,\rho'(s,x))\|_{\Rnu^P \x \ell_2}
\notag\\
&+ \sup_{x'\in\Rnu^P}|\nab\nab \te(t,x')| \Big(\int_\Rnu
|\dl(s,x,\rho(t,x),y) - \dl(s,x,\rho'(t,x),y)|^2
\nu(dy)\Big)^\frac12\Big].
}
By \rf{dl-claim}, the last summand in \rf{gen-diff} is 
smaller than $$\|\sg(s,x,\rho(s,x))-\sg(s,x,\rho'(s,x))\|_{\Rnu^P \x
\ell_2}$$ up to a multiplicative constant. 
Therefore, modifying the constant
$\bar K$, if necessary, by Assumption A1, we obtain that \aa{
\sup_{\|\te\|_D \lt 1} \|A(s,\rho(s, \fdot))\te - A(s,{\rho'}(s,
\fdot))\te\|_E \lt \bar K \|\rho(s,\fdot) - \rho'(s,\fdot)\|_E }where 
$\bar K$ does not depend on $s$, $\rho$, and $\rho'$.
Now by \rf{02-08}, there exists a constant $K>0$, so that
\aa{ \sup_{s\in [t,T]} \|U(t,s,\rho)\ffi -
U(t,s,\rho')\ffi\|_\e \lt K\,T \hspace{-1mm} \sup_{s\in
[t,T]}\|\rho(s,\fdot) - \rho'(s,\fdot)\|_\e\, \|\ffi\|_{D}.}
Let us show that $K$ does not depend on $t$, $s$, $\rho$, and $\rho'$.
By It\^o's formula, for $s\in [t,T]$ and for $\ffi\in D$,
\aaa{
\lb{for22}
[U(t,s,\rho)\ffi](x) =  E[\ffi(X_s)|X_t = x],
}
where $X_s$ is the solution to \aa{ 
dX^k_s = f^k(s,X_s,\rho(s,X_s)) +
\sum_{i=1}^\infty \sg^k_i(s,X_{s-},\rho(s,X_{s-}))dH^\i_s. }
Moreover, by the results of \cite{Tanabe} (p. 102),  
$U(t,s,\rho)$ maps $D$ into $D$, 
and \rf{for22} implies that $U(t,s,\rho)\in \mc L(D)$ so that
the norm $\|U(t,s,\rho)\|_{\mc L(D)}$ is bounded uniformly in $\rho$.
Next, since  for each $\ffi\in D$, $U(t,s,\rho)\ffi \in D$ is 
continuous in $t$ and $s$, then it is bounded uniformly in $t$ and $s$.
Therefore, $\|U(t,s,\rho)\|_{\mc L(D)}$ is bounded uniformly in $t$, $s$, and $\rho$.
This implies the statement of the lemma.
\end{proof}
\begin{Thm}
Let Assumptions A0 and A1 hold. 
Then, there exists a $T_0>0$ so that for all
$T\in (0,T_0]$, Problem
\rf{evolution} has a unique solution on $[0,T]$.
\end{Thm}
\begin{proof}
Consider the equation:
\aaa{
\label{evolution3} \te(t,x) = [U(t,T,\te)h](x) + \int_t^T
[U(t,s,\te)\, g(s,\fdot,\te(s,\fdot))](x)\, ds. }
The proof of the existence and uniqueness of a solution to
\rf{evolution3} is equivalent to the existence of a unique fixed
point of map \rf{map-phi} in the space $E$. For a sufficiently small time interval
$[0,T]$, the latter is implied by Assumption A1 and Lemma \ref{fplem}.
Now let $\te$ be the solution to \rf{evolution3} on $[0,T]$. Consider the equation
\aaa{
\lb{evolution4} \bar\te(t,x) = [U(t,T,\te)h](x) + \int_t^T
[U(t,s,\te)\, g(s,\fdot,\bar\te(s,\fdot))](x)\, ds }
in the space $D$. Since $\|U(t,s,\te)\|_{\mc L(D)}$ is bounded, and $g(s,x,y)$
is Lipschitz in $y$ whose Lipschitz constant does not depend on $s$ and $x$,
the fixed point argument implies the existence of a unique solution $\bar \te \in D$
to \rf{evolution4}.
Clearly, $\bar \te$ is also a unique solution to \rf{evolution4} in $E$.
Hence $\bar \te = \te$, and therefore, $\te\in D$. This implies
that $\te$ is the unique solution to Problem \rf{evolution}.
\end{proof}
\subsection*{3.3 Existence and Uniqueness Theorem for the {FBSDEs}}
In Paragraph 3.2 we found some conditions under
which there exists a unique solution to PIDE \rf{PIDE}. However,
this solution may exist under more general assumptions. Thus, we
prove the existence and uniqueness of a solution to FBSDEs
\rf{fbsde1} assuming the existence and uniqueness of a solution to
PIDE \rf{PIDE}. Specifically, we will assume the following:
\begin{enumerate}
\item[\bf A2] Functions $f$, $g$, and $\sg$ possess bounded first
order derivatives in all spatial variables.
\item[\bf A3] Assumption A0 is fulfilled and Final value problem \rf{PIDE} has a unique solution $\te$
which belongs to the class $\C_b^{1,2}([0,T]\x \Rnu^P\to\Rnu^Q)$.
\item[\bf A4] There exists a constant  $K>0$ which does not depend on $(t,x,y,z)$,
such that $\sum_{i=1}^\infty \big|\frac{\pl}{\pl z_i}
f(t,x,y,\{z_i\}_{i=1}^\infty)\big| \big(\int_\Rnu
|p_i(y)|^2\nu(dy)\big)^\frac12 < K$.
\end{enumerate}
\begin{lem}
Assume A2, A3, and A4 hold. Then the function $f(t,\bar x, \bar y,
\fdot) \circ \te^{(1)}(t,x)$, where $\te^{(1)}(t,x)$ is given by
\rf{te1}, is Lipschitz in $x$ for all $(t,\bar x,\bar y)$, and the
Lipschitz constant does not depend on $(t,\bar x,\bar y)$.
\end{lem}
\begin{proof}
Note that by Assertions \ref{cik1} and \ref{asr2} of Lemma
\ref{assers}, the function $c^k = \{c^k_i\}_{i=1}^\infty$ is
Lipschitz in two spatial variables as an $\ell_2$-valued function.
By A3, $\te$ and $\pl_k\te$ are Lipschitz. Therefore, the last
summand in \rf{te1} is Lipschitz in $x$, and moreover, its
Lipschitz constant does not depend on $t$ by boundedness of the both
multipliers. 
Let us prove that the map  
\aaa{ \lb{mp1} \Rnu^P \to \Rnu^Q,  \quad x
\mto \int_\Rnu \te(t,\bar x + \dl(t,x,\rho(t,x),y))\, p_i(y)\nu(dy)
} is Lipschitz, where $\bar x$  and $t$ are fixed. Let $x_1,x_2 \in \Rnu^P$, 
and let $\rho_1=\rho(t,x_1)$ and $\rho_2 =
\rho(t,x_2)$, where $t$ is fixed. We have:
\mmm{ \lb{arg55}
\Big|\int_\Rnu [\te(t,\bar x + \dl(t,x_1,\rho_1,y)- \te(t,\bar x + \dl(t,x_2,\rho_2,y)] p_i(y)\nu(dy)\Big|\\
\lt \max_{x\in\Rnu^P} |\nab \te(t,x)| \int_\Rnu
|\dl(t,x_1,\rho_1,y) - \dl(t,x_2,\rho_2,y)|\,|p_i(y)|\nu(dy)\\
\lt \max_{x\in\Rnu^P} |\nab \te(t,x)|\Big(\int_\Rnu
|\dl(t,x_1,\rho_1,y) - \dl(t,x_2,\rho_2,y)|^2 \nu(dy)
\int_\Rnu |p_i(y)|^2\nu(dy)\Big)^\frac12\\
\lt K \max_{x\in\Rnu^P} |\nab \te(t,x)|\Big(\int_\Rnu
|p_i(y)|^2\nu(dy)\Big)^\frac12 \|\sg(t,x_1,\rho_1) -
\sg(t,x_2,\rho_2)\|_{\Rnu^P\x \ell_2}. }
Now the Lipschitzness of map \rf{mp1} and the boundedness of the gradient of $\te$ imply
that the map \aa{ \Phi: \Rnu^P \to \Rnu^Q \x \ell_2, \;
x\mto\int_\Rnu \te(t,x + \dl(t,x,\theta(t,x),y)) - \te(t,x))
p_i(y)\, \nu(dy) } is also Lipschitz. Argument \rf{arg55} implies
that the Lipschitz constant of $\Phi$ has the form $\td
K\,\Big(\int_\Rnu |p_i(y)|^2\nu(dy)\Big)^\frac12$ where $\td K$ is a
constant that does not depend on $i$. Now A4 implies the statement
of the lemma.

\end{proof}
\begin{Prop}
Assume A2, A3, and A4. Then, the SDE
\begin{gather}
\label{sdetheta}
\left\{
\begin{array}{l}
dX_t = f(s,X_s,\theta(s,X_s),\theta^{(1)}
(s,X_{s-})) ds
 + \sum_{i=1}^\infty\sg_i(s,X_{s-},\te(s,X_{s-})) dH^\i_s,\\
X_0 = x,
\end{array}
\right.
\end{gather}
 where $\te$ is the solution to \rf{PIDE} and
$\te^{(1)}$ is defined by \rf{te1}, has a pathwise unique 
c\`adl\`ag adapted solution.
\end{Prop}
\begin{proof}
We will show that 
$$ \Psi(X)_t= x + \int_0^t
f(s,X_s,\theta(s,X_s),\theta^{(1)}(s,X_{s-})) ds
 + \int_0^t \sg(s,X_{s-},\theta(s,X_{s-}))dH_s
$$ 
is \ \ a \ \ contraction \ \ map \ \ in \ \ the \ \ Banach \ \ space
\ \ $S$ \ \ with \ \ the \ \ norm \ \ $\|\Phi\|^2_S =
E\sup_{t\in[0,T]} |\Phi_t|^2$. Take two points $X_s$ and $X'_s$ from
$S$. For simplicity of notation, let $\sg_s =
\sg(s,X_{s-},\te(s,X_{s-}))$ and $\sg'_s =
\sg(s,X'_{s-},\te(s,X'_{s-}))$. To estimate the difference of the
stochastic integrals with the integrands $\sg_s$ and $\sg'_s$ with
respect to the $\|\fdot\|_S$-norm, we apply the
Burkholder--Davis--Gundy inequality to the martingale $\int_0^t
(\sg_s - \sg'_s) dH_s$. We obtain that there exists a constant $C>0$
such that \mm{
E\sup_{r\in[0,t]}\Big|\int_0^r (\sg_s - \sg'_s) dH_s\Big|^2
\lt C E \Big[\int_0^\bullet (\sg_s - \sg'_s)dH_s\Big]_t\\
=C E\Big(\Big\< \int_0^\bullet (\sg_s - \sg'_s)dH_s\Big\>_t +
U_t\Big)
= C E \sum_{i,j=1}^\infty \int_0^t (\sg_i-\sg'_i,\sg_j-\sg'_j)d\<H_i,H_j\>_s\\
= C E \int_0^t \|\sg_s-\sg'_s\|^2_{\Rnu^P\x \ell_2} ds } where
$[\fdot]_t$ and $\<\fdot\>_t$ are the quadratic variation and the
predictable quadratic variation, respectively. Moreover, we applied
the identity $\<H_i,H_j\>_s = \dl_{ij} s$ and the decomposition
$[M]_t = \<M\>_t + U_t$ for the quadratic variation of a square
integrable martingale
(i.e. a martingale $M_t$ with $\sup_t |M_t|^2 <\infty$)
into the sum of the predictable quadratic variation and a uniformly
integrable martingale $U_t$ starting at zero. Next, we note that the
functions $x\mto f(s,x,\te(t,x),\te^1(t,x))$ and $x\mto
\sg(t,x,\te(t,x))$ are Lipschitz whose Lipschitz constants do not
depend on $t$. This and the above stochastic integral estimate imply
that there exist a constant $K>0$ such that 
$$E \sup_{s \in [0,t]}  |\Psi(X)_s -\Psi(X')_s|^2  \lt K  E \int_0^t
|X_s-X'_s|^2 ds \lt K  E \int_0^t \sup_{r\in[0,s]}|X_r-X'_r|^2 ds.
$$ 
Iterating this $n-1$ times we obtain: \aa{ E \sup_{s \in [0,t]}
|\Psi^n(X)_s -\Psi^n(X)'_s| \lt \frac{K^n t^n}{n!}
E \sup_{ s \in [0,t]} |X_s-X'_s|^2. }
Choosing $n$ sufficienty large so that $\frac{K^nT^n}{n!} < 1,$ we
obtain that $\Psi^n$ is a contraction, and thus, $\Psi$ is a
contraction as well. By the Banach fixed point theorem, the map
$\Psi$ has a unique fixed point in the space $S$. Clearly, this
fixed point is a unique solution to \rf{sdetheta}. Setting $X^{(0)}
= x$, and then, sucessively, $X^{(n)} = \Psi(X^{(n-1)})$, we can
choose c\`adl\`ag modifications for each  $X^{(n)}$. Since the
$X^{(n)}$'s converge to the solution $X$ in the norm of $S$, $X$
will be also c\`adl\`ag a.s.. This c\`adl\`ag solution is unique in the
space $S$, and therefore, pathwise unique.
\end{proof}
Introduce the space $\mc S$ of $\mc F_t$-predictable $\Rnu^Q\x
\ell_2$-valued stochastic processes with the norm $\|\Phi\|_{\mc
S}^2 = E\int_0^T \|\Phi_s\|^2_{\Rnu^Q\x \ell_2} ds$. Now we
formulate our main result.
\begin{Thm}
\lb{main-result} Suppose A2, A3, and A4 hold.  Let $X_t$ be the
c\`adl\`ag adapted solution to \rf{sdetheta}. Then, the triple
$(X_t,Y_t,Z_t)$, where $Y_t=\te(t,X_t)$, $Z_t= \te^{(1)}(t,X_{t-})$
with $\theta^{(1)}$ given by \rf{te1}, is a solution to FBSDEs
\rf{fbsde1}. Moreover, the pair of c\`adl\`ag solution processes
$(X_t,Y_t)$ is pathwise unique. The solution process $Z_t$ is unique
in the space $\mc S$.
\end{Thm}
\begin{proof}
It suffices to prove that the triple $(X_t,Y_t,Z_t)$ defined in the
statement of the theorem verifies the BSDE in \rf{fbsde1}.
Application of It\^o's formula to $\te(t,X_t)$ gives: 
\mmm{
\label{itotheta}
\theta(T,X_t) -\theta(t,X_t) = \int_t^T \partial_s \theta(s,X_{s-})ds +\int_t^T \partial_{k}\theta(s,X_{s-})dX_s^k \\
+ \frac12 \int_t^T \partial^2_{kl}
\theta(s,X_{s-})d[(X^c)^k,(X^c)^l]_s\\ + \sum_{t < s \lt T}[
\theta(s,X_s)-\theta(s,X_{s-})- \Delta X^k_s \partial_{k}
\theta(s,X_{s-})], } where $X^c_s$ is the continuous part of $X_s$.
Using the representation for $H^\i_s$ from Lemma \ref{hi-rep}
we obtain that \aa{ d[(X^c)^k,(X^c)^l]_s= 2
\beta^{kl}(s,X_s,\theta(s,X_s))ds,
} where $\beta^{kl}$ is given by \rf{akl}. The forward SDE in
\rf{fbsde1}, the relation $\Delta H^\i_s = p_i(\Delta L_s)$,
obtained in \cite{nualart}, and representation \rf{dl} for the
function $\dl$ imply: \aaa{ \lb{Dlx} \Delta X_s =
\sum_{i=1}^{\infty} \sigma_i(s,X_{s-},Y_{s-}) \Delta H_s^\i =
\dl(s,X_{s-},Y_{s-}, \Dl L_s). }
Next, one can rewrite the last term in (\ref{itotheta}) as \mm{
\sum_{t < s \leq T}\Big[ \theta(s,X_{s-}+ \dl(s,X_{s-},Y_{s-}, \Dl
L_s))-\theta(s,X_{s-}) \\ - \dl^k(s,X_{s-},Y_{s-}, \Dl
L_s)\,\partial_{k} \theta(s,X_{s-})\Big], } where $\dl^k$ is the
$k$th component of $\dl$. Define the random function \aaa{ \lb{fh}
h(s,y)&= \theta(s,X_{s-} + \delta(s,X_{s-},\theta(s,X_{s-}),y)) - \theta(s,X_{s-})\notag\\
& -\dl^k(s,X_{s-},\theta(s,X_{s-}), y)\,\pl_{k} \te(s,X_{s-}). }
Note that for each fixed $s\in [0,T]$ and $\om\in \Om$, the function
$h$ satisfies condition \rf{condition5}. Indeed, the mean value
theorem, e.g. in the integral form, can be applied to the difference
of the first two terms in \rf{fh}. By boundedness of the partial
derivatives $\pl_k\te$, it suffices to verify that \aa{ E\int_0^T
|\dl(s,X_{s-},\theta(s,X_{s-}), y)|^2\, ds \, \nu(dy) < \infty. }
The latter holds by Assumption A0 and formula \rf{dl-claim}. Now
Lemma~\ref{lemma5} implies: \aa{
&\sum_{t < s \leq T} h(s,\Delta L_s)\\
&=\sum_{i=1}^{\infty}\int_t^T\int_\RR \Big[ \theta(s,X_{s-} + \delta(s,X_{s-},\theta(s,X_{s-}),y)) - \theta(s,X_{s-})\\
& -\dl^k(s,X_{s-},\te(s,X_{s-}),y) \,\pl_{k} \te (s,X_{s-})\Big]p_i(y) \nu(dy)dH_s^{\i}\\
&+ \int_t^T\int_\RR \Big[\theta(s,X_{s-} + \delta(s,X_{s-},\theta(s,X_{s-}),y)) - \theta(s,X_{s-})\\
&  -\dl^k(s,X_{s-},\te(s,X_{s-}),y) \,\pl_{k} \te
(s,X_{s-})\Big]\nu(dy)\, ds. }
Substituting this into \eqref{itotheta}, replacing $dX_s^k$ with the
right-hand side of \rf{sdetheta}, and taking into acccount that
$Y_t= \theta(t,X_t)$ and that $\theta(T,X_T)= h(X_T)$ by \rf{PIDE},
we obtain: \aa{ Y_t&= h(X_T) - \int_t^T \bigg[ \partial_s
\theta(s,X_{s-})+
\partial_{k} \theta(s,X_{s-})f^k(s,X_{s-},\theta(s, X_{s-}),\te^{(1)}(s,X_{s-}))\\
&+ \frac12 \partial_{kl} \theta(s,X_{s-})
\beta^{kl}(s,X_{s-},\theta(s,X_{s-}))\\
&+ \int_\RR  \big[\theta(s,X_{s-} + \delta(s,X_{s-},\theta(s,X_{s-}),y))\\
&- \theta(s,X_{s-}) -\dl^k(s,X_{s-},\te(s,X_{s-}),y) \,\pl_{k} \te (s,X_{s-})\big]\nu(dy)\bigg] ds\\
& - \int_t^T \sum_{i=1}^{\infty} \bigg[ \int_\RR
\big[\theta(s,X_{s-} + \delta(s,X_{s-},\theta(s,X_{s-}),y)) -\theta(s,X_{s-}) \\
&- \partial_{k}  \theta(s,X_{s-}) c_i^k (X_{s-},y)\big]p_i(y)
\nu(dy) \bigg]\: dH_s^\i. }
Clearly, in the first three summands under the $ds$-integral sign
one can equivalently write $X_s$ or $X_{s-}$. This is true since
$X_s$ has c\`adl\`ag paths, and therefore, $X_s$ and $X_{s-}$ can
differ only at a countable number of points. Now taking into account
PIDE (\ref{PIDE}), we note that the integrand in the drift term is
$-g(s,X_{s-},\theta(s,X_{s-}), \theta^{(1)}(s,X_{s-}))$ which is
$-g(s,X_{s-},Y_{s-},Z_s)$ by the definitions of $Y_s$ and $Z_s$, or,
it can be replaced by $-g(s,X_{s},Y_{s},Z_s)$ since $X_s$ and $Y_s$
have c\`adl\`ag paths. Finally, by \rf{te1} and the definition of
$Z_s$, the integrand in the stochastic term is $Z_s$.
Consequently, \aa{ Y_t = h(X_T) + \int_t^T g(s,X_s,Y_s,Z_s) -
\int_t^T Z_s \: dH_s, } which implies that $(X_s,Y_s,Z_s)$ is a
solution.

Let us prove the uniqueness. Let $(X_s,Y_s,Z_s)$ be an arbitrary
solution to \rf{fbsde1}. Let $\til{Y}_s= \theta(s,X_s),$ and
$\til{Z}_s = \theta^{(1)}(s,X_{s-})$, where $\theta$ is the solution
to \rf{PIDE}, and $\theta^{(1)}$ is defined by \rf{te1}. By the
above argument, $(X_s, \til{Y}_s, \til{Z}_s)$ verifies the BSDE in
\rf{fbsde1}. Applying It\^o's product formula to
$|\tilde{Y}_t-Y_t|^2$ and taking into consideration that $\td Y_T =
Y_T$, we obtain: \aa{ |\til{Y}_t - Y_t|^2  = -2 \int_t^T \Big(
\til{Y}_{s-}-Y_{s-}, d(\tilde{Y}_s-Y_s)\Big) +[\td Y - Y]_t - [\td Y
- Y]_T. }
Taking the expectations in the above relation gives: \aa{
&E|\tilde{Y}_s-Y_s|^2 +  E \int_t^T \|\til{Z}_s - Z_s\|_{\Rnu^Q \x \ell_2}^2 ds \\
&=2E \int_t^T \big(\til{Y}_s-Y_s,
g(s,X_s,\til{Y}_s,\til{Z}_s)-g(s,X_s,Y_s,Z_s)\big) \, ds. } By A2,
there exists a constant $C>0$ such that \mm{
E |\til{Y}_t- Y_t|^2 + E \int_t^T \|\til{Z}_s-Z_s\|_{\Rnu^Q\x\ell_2}^2\, ds \\
\lt C E \int_t^T |\til{Y}_s- Y_s|\big( |\til{Y}_s- Y_s| +
\|\til{Z}_s-Z_s\|_{\Rnu^Q\x\ell_2}\big) \, ds. }
%
Now using the standard estimates and applying Gronwall's inequality,
we obtain that $E |\til{Y}_t- Y_t|^2 + c E \int_t^T
\|\til{Z}_s-Z_s\|_{\Rnu^Q\x \ell_2}^2 ds = 0$ for some constant
$c>0$. The latter relation holds for all $t\in [0,T]$. This proves
that $\td Y_t$ is a modification of $Y_t$ and that $\|\td Z-Z\|_{\mc
S} = 0$. This implies the uniqueness result. 
\end{proof}

\,%
%
\section*{4 Option Pricing with a Large Investor in \\ L\'evy-type Markets}
\setcounter{section}{4}

\setcounter{equation}{0}

\setcounter{Thm}{0}

 Usually, when modeling financial assets it is assumed that all investors are price
takers whose individual buy and sell decisions do not influence the
price of assets. Cvitanic and Ma~\cite{cvitanic} have already
developed a model for hedging options in the presence of a large
investor in a Brownian market. However, observation of real data
suggests that patterns, like skewness,
 kurtosis, or the occurence of jumps are sufficiently significant
(see, e.g., Eberlein and Keller \cite{eberlein}) to deserve to be
accounted in a realistic model of option pricing. Furthermore, the
graphs of the evolution of stock prices at different time-scales are
sufficiently different from the self-similarity of a Brownian
motion. Thus, we develop a L\'evy-FBSDE option pricing model. We
believe that such a model conveys a much more realistic approach to
option pricing  in the presence of the already mentioned empirical
market characteristics. We assume the existence of a Large investor,
whose wealth and strategy may induce distortions of the price
process.

Let ${\mathcal M}$ be a  L\'evy-type Market, i.e. a market whose
stock price dynamics $S_t$ obeys the equation $S_t=S_0 e^{X_t}$,
where $X_t$ is a L\'evy-type stochastic integral \cite{applebaum}.
The market consists of $d$ risky assets and a money market account.
For the price process $P_0(t)$ of the money market account, we
assume that its evolution is given by the following equation \aa{
&dP_0(t)=P_0(t)\, r(t,W(t), Z(t)) \,dt, \:\: 0 \lt t \lt T ,\\
&P_0(0)= 1, } where $W(t)$ is the \emph{wealth process}, and $Z(t)$
is a portfolio-related process in a way that will be explained
later. For the risky assets, we add the stochastic component
represented by the volatility matrix $\sigma$ taking values in
$\RR^d \times \ell_2$.
We postulate that the evolution of the $d$-dimensional risky asset
price process $P(t)=\{P_i(t)\}_{i=1}^d$  is given by the following
SDE: \aaa{ dP_i(t)&= f_i(t,P(t),W(t), Z(t)) \:dt
+  \sum_{j=1}^{\infty}  \sigma_{j}^i(t,P(t), W(t)) \: dH^\j_t, \nonumber \\
P_i(0)&= p_i,\; \; p_i\gt 0,\; \; 1 \lt i \lt d, \quad t\in [0,T].
\label{sdepi} }
We derive the BSDE for the wealth process as in \cite{cvitanic}. For
the convinience of the reader we repeat this derivation: \aa{dW(t)=
\sum_{i=1}^{d} \al_i(t)\, dP_i(t) + \frac{W(t)- \sum_{i=1}^{d}
\al_i(t)P_i(t)}{P_0(t)} \,dP_0(t), } where $\al_i(t)$ is the
\emph{portfolio process}.
Substituting $dP_i(t)$ with the right-hand sides of (\ref{sdepi}),
we obtain:
\mmm{
\lb{wp}
dW(t) =\sum_{i=1}^{d} \al_i(t) \big[
f_i(t,P(t),W(t), Z(t)) dt
+  \sum_{j=1}^{\infty}  \sigma^i_j(t,P(t), W(t)) \, dH^\j(t)\big]\\
+ (W(t)- \sum_{i=1}^{d} \al_i P_i(t))\,r(t,W(t),Z(t)) \,dt\\
=g(t,P(t),W(t),Z(t), \al(t))dt + \sum_{i=1}^{\infty} Z_i(t)
dH^\i(t), 
}
where \aaa{ &g(t,\pi,w,z,a)
=\sum_{i=1}^{d} a_i f_i(t,\pi,w, z)
+ (w- \sum_{i=1}^{d}a_i \pi_i)\,r(t,w, z),\notag\\
&a = \{a_i\}_{i=1}^d,  \; \pi = \{\pi_i\}_{i=1}^d;\notag\\
&Z_i(t)=\sum_{j=1}^{d} \al_j(t)  \sigma_i^j(t,P(t), W(t)), \quad
i=1,2, \ldots. \lb{z-al} } As we are assuming the absence of risk
for the money market account, the evolution of its price depends
totaly on the interest rate the investor is earning. On the other
hand, to describe the evolution of the risky assets we use an SDE
with the stochastic term given as a sum of stochastic integrals with
respect to $H^\j$'s. This adds explanative power to the model, as it
affords the isolation of the individual contributions of each
$H^\j$.
Now, to guarantee that the stock price has positive components we
will rewrite \rf{sdepi} for $Q_i(t) = \log P_i(t)$ using It\^o's
formula.
For simplicity of notation, we will use the same symbols $f$, $g$,
$\sg$, and $h$ for the coefficients of the FBSDEs which we obtain
after rewriting SDE \rf{sdepi} with respect to
$Q(t)=\{Q_i(t)\}_{i=1}^d$ and substituting $P_i(t) =
\exp\{Q_i(t)\}$: \aaa{ \lb{sdeq} Q(t)= q + \int_0^t f(s,Q(s),W(s),
Z(s)) \,ds + \int_0^t \sigma(s,Q(s), W(s)) \, dH(s), } where $q =
\{\log p_i\}_{i=1}^{d}$. Due to relation \rf{z-al}, we exclude the
dependence on $\al(t)$ in \rf{wp}. BSDE \rf{wp} takes the form:
\aaa{ \lb{bsdew} W(t)= h(Q(T)) + \int_t^T g(s,Q(s),W(s),Z(s)) \, ds
- \int_t^T  Z(s)\, dH(s). }
Theorem \ref{main-result} and relation \rf{z-al} imply the following
result.
\begin{Thm}
Assume A2, A3 and A4. Then, FBSDEs (\ref{sdeq}-\ref{bsdew})
has a unique solution $(Q(t),W(t),Z(t))$ such that the pair
$(Q(t),W(t))$ is c\`adl\`ag. Furthermore, if for some
$\Rnu^d$-valued stochastic process $\{\al_j(t)\}_{j=1}^d$,
$\al_j(t)\gt 0$, relation  \rf{z-al} holds in the space $\mc S$,
then $\{\al_j(t)\}_{j=1}^d$ is a replicating portfolio.
 \end{Thm}
It is evident that L\'evy-type markets pose theoretical questions
that had never been raised in the Brownian motion framework.
 We conclude by reenforcing the idea that the
impossibility of replicating every potential contingent claim is
rather an expected characteristic due to the complexity of the price
formation in L\'evy-type markets, than a drawback of our model.

{\bf Acknowledgements.} The research of E. Shamarova was funded by
the European Regional Development Fund through the program COMPETE
and by the Portuguese Government through the FCT (Funda\c{c}\~ao
para a Ci\^encia e a Tecnologia) under the project
PEst-C/MAT/UI0144/2011. R.\,S. Pereira was supported by the Ph.D.
grant SFRH7BD/51172/2010 of FCT. The authors thank Wolfgang Polasek
for useful comments.
%
%
%
%

\end{document}